\documentclass[twoside,11pt,reqno]{amsart}
\usepackage{amscd, bbm}
\usepackage{amssymb,mathabx}
\usepackage{todonotes}
\usepackage{mathrsfs,wasysym, amscd}
\usetikzlibrary{matrix, arrows}

\makeatletter

 \usepackage{geometry}
\@addtoreset{equation}{section}

\newtheorem{Theorem}{Theorem}
\newtheorem{Lemma}[Theorem]{Lemma}
\newtheorem{Proposition}[Theorem]{Proposition}

\newtheorem*{Theoremm}{Theorem}

\theoremstyle{remark}
\newtheorem{Remark}[Theorem]{Remark}

\newtheorem*{Question}{Question}

\newbox\squ  
\setbox\squ=\hbox{\vrule width.6pt
   \vbox{\hrule height.6pt width.4em\kern1ex
         \hrule height.6pt}%
   \vrule width.6pt}

\def\gr{\operatorname{gr}}

\def\ad{{\operatorname{ad}}}

\def\ind{{\operatorname{ind}\,}}

\def\Der{{\operatorname{Der}}}

\def\Ker{{\operatorname{Ker}}}

\def\GK{{\operatorname{GKdim}\,}}

\newcommand{\oL}{{\overline{L}}}

\def\k{{\mathbbm k}}

\def\A{{\mathcal A}}
\def\0{{\bar 0}}
\def\1{{\bar 1}}

\newdimen\Hoogte    \Hoogte=10pt    
\newdimen\Breedte   \Breedte=10pt   
\newdimen\Dikte     \Dikte=0.5pt    

\begin{document}
\title{A non-restricted counterexample to the first Kac--Weisfeiler conjecture}
\author{Lewis Topley}
\maketitle

\begin{abstract}
In 1971 Kac and Weisfeiler made two important conjectures regarding the representation theory of restricted
Lie algebras over fields of positive characteristic. The first of these predicts the maximal dimension of the
simple modules, and can be stated without the hypothesis that the Lie algebra is restricted. In this short article
we construct the first example of a non-restricted Lie algebra for which the prediction of the first Kac--Weisfeiler conjecture fails.
Our method is to present pairs of Lie algebras which have isomorphic enveloping algebras but distinct indexes.
\end{abstract}


\section{Introduction}

Let $\k$ be an algebraically closed field of positive characteristic $p$ and $L$ be a finite dimensional Lie algebra over $\k$.
It is well known that all simple modules have finite dimension and that the dimensions are uniformly bounded above by some integer.
We denote by $M(L)$ the least upper bound of dimensions of simple $L$-modules. We remind the reader that
the index of $L$, denoted $\ind L$, is the minimal dimension of a stabiliser of an element of the coadjoint
representation. The number $\dim L-\ind L$ is easily seen to be even and the first Kac--Weisfeiler conjecture (KW1)
predicts that the index of a restricted Lie algebra $L$ is involved in the representation theory in the following way:
\begin{eqnarray}\label{KWSt}
M(L) = p^{\frac{1}{2}(\dim L - \ind L)}
\end{eqnarray}

\noindent \cite[\S 1.2]{KW71}. The conjecture is striking for both its simplicity and its generality, and has attracted much attention over the past 45 years.
Since the statement may be phrased without the hypothesis that $L$ is restricted, there has been some small hope that it may hold in general.
In this paper we shall show that for certain non-restricted Lie algebras (\ref{KWSt}) fails. These are the very first examples of this kind in the literature.

For a given Lie algebra $L$ the problem of calculating $\ind L$ belongs to the realm of elementary linear algebra
and the meat of the KW1 conjecture lies in computing $M(L)$. There is no procedure for determining this invariant
in general, and practically nothing is known about representations of Lie algebras which are not restricted,
which is undoubtedly why it has taken so long for (\ref{KWSt}) to be refuted for non-restricted algebras.
The most general result appears in \cite[Thm. 4.4]{PS99} where it is shown that if a restricted Lie algebra $L$
admits a $\chi \in L^*$ such that $L_\chi$ is a torus then KW1 holds for $L$.

Over the past 10 years, various authors have been studying the isomorphism problem for enveloping algebras (see \cite{RU07} for example).
In its most general form, the question is: \emph{can two non-isomorphic Lie algebras admit isomorphic
enveloping algebras?} For finite dimensional Lie algebras over fields of characteristic zero there are no known
examples of this pathalogical behaviour, however in characteristic $p$ such algebras are not hard to construct
(we shall see new examples of this phenomenon in Proposition~\ref{isoprop}). Several weaker variants of the isomorphism problem have been considered,
asking which properties are shared by Lie algebras $L$ and $L'$ such that $U(L) \cong U(L')$, for instance
nilpotence, solvability, derived length. The key observation of this article is that (\ref{KWSt}) implies a weak variant of the isomorphism
problem: if (\ref{KWSt}) holds for all $\k$-Lie algebras and $U(L) \cong U(L')$ then $\ind L = \ind L'$.
This is simply because both $M(L)$ and $\dim(L)$ depend only upon the isomorphism class of $U(L)$;
in the language of \cite{RU07} we would say that $\ind L$ is determined by $U(L)$.
Our method is to disprove this corollary of (\ref{KWSt}) by exhibiting two Lie algebras with isomorphic enveloping algebras
but distinct indexes.

For any set $X$ we use the notation $\langle X \rangle$ to denote the vector space spanned by $X$.
We now describe a family of examples for which (\ref{KWSt}) fails. Let $k \geq 3$ and let $L$ be the Lie algebra
$\langle x_1, x_2, ... x_k, D_0, D \rangle$ such that $D_0$ is central, $\langle x_1,... x_{k}\rangle$ is abelian, whilst
\begin{eqnarray*}
& & [D, x_i] = x_i \text{ for } i=1,...,k-2,\\
& & {[}D, x_{k-1}] = x_k,\\
& & {[}D, x_k] = 0.
\end{eqnarray*}
In this article we shall prove that:
\begin{Theoremm}
We have $p^2 | M(L)$ and $\ind L = k$ so that $M(L) \neq p^{\frac{1}{2}(\dim(L) - \ind(L))}$.
\end{Theoremm}
\noindent Note that the example above is not restrictable, since $\ad(D)^p \notin \ad(L)$.

\begin{Question}\label{KWQ}
\begin{enumerate}
\item[i)]{Does there exist a restricted Lie algebra for which KW1 fails?}
\item[ii)]{Do there exist two restricted Lie algebras $L$ and $L'$ with $U(L) \cong U(L')$ and $\ind L \neq \ind L'$?}
\end{enumerate}
\end{Question}
A positive answer to question (ii) would imply a positive answer to question (i), whilst a negative answer to (ii)
would offer supporting evidence for the KW1 conjecture, as well as having independent value in the context of the isomorphism problem.

\noindent {\bf Acknowledgement.} I would like to thank Stephane Launois and Alexander Premet for useful comments on a first draft of this article. I would also
like to thank colleagues Giovanna Carnovale and Jay Taylor at the University of Padova for many interesting discussions. The research
leading to these results has received funding from the European Commission, Seventh Framework Programme, under Grant Agreement
number 600376, as well as grants CPDA125818/12 and 60A01-4222/15 from the University of Padova.

\section{Lie algebras with isomorphic enveloping algebras}
In this section we prove a basic result which allows us to construct families of 
Lie algebras which have isomorphic enveloping algebras. For any Lie algebra $L$ we consider the
restricted closure $\oL$ of $\ad(L)$ inside $\Der(L)$, ie. the smallest restricted subalgebra of $\Der(L)$
containing $\ad(L)$.
\begin{Lemma}
Every element of $\oL$ is of the form $\sum_{i=0}^k \ad(X_i)^{p^i}$ for some $k \geq 0$ and elements $X_1,...,X_k \in L$.
\end{Lemma}
\begin{proof}
We start by showing that for each $k \geq 0$ the sum $\sum_{i=0}^k \ad(L)^{p^i}$ is a vector space.
The case $k = 0$ is obvious so we may proceed by induction. Using the formulas derived in Chapter~$2$ of \cite{FS88} we have
$\ad(X  + Y)^{p^k} = \ad(X)^{p^k} + \ad(Y)^{p^k} \mod \ad(L)$ and so
$$\sum_{i=0}^k(\ad(X_i)^{p^i} + \ad(Y_i)^{p^i}) = \sum_{i=0}^{k-1} (\ad(X_i)^{p^i} + \ad(Y_i)^{p^i}) + \ad(X + Y)^{p^k} \mod \ad(L)$$
and so by induction $\sum_{i\geq 0} \ad(L)^{p^i}$ is a vector space.

If $X_1 = \ad(X)$ and $Y_1 = \ad(Y)$ then $[X_1^{p^k}, Y_1^{p^j}] = \ad(X_1)^{p^k-1}(\ad(Y_1)^{p^j} X_1$ and so $\sum_{i\geq 0} \ad(L)^{p^i}$
is closed under the bracket. Using the same formulas mentioned in the first paragraph of the proof it is clear that $\sum_{i\geq 0} \ad(L)^{p^i}$
is closed under taking $p$th powers, and so it is a restricted algebra containing $\ad(L)$. It is easy to see that it is the smallest such algebra.
\end{proof}

For $D \in \Der(L)$ we write $L_D$ for the semidirect product $L \rtimes \k D$.

\begin{Proposition}\label{isoprop}
For every $D, D' \in \oL$ we have $U(L_D) \cong U(L_{D'})$.
\end{Proposition}
\begin{proof}
Let $D_0$ denote the zero derivation of $L$. We shall show that $U(L_D) \cong U(L_{D_0})$ for every
$D\in \oL$. According to the previous lemma we can write
$D = \sum_{i=0}^k \ad(X_i)^{p^i}$ for some $k \geq 0$ and elements $X_1,...,X_k \in L$.

We define a linear map
\begin{eqnarray*}
\phi& : & L_{D} \lhook\joinrel\longrightarrow U(L_{D_0});\\
& & L \overset{\operatorname{Id}}{\longmapsto} L;\\
& & D \longmapsto \sum_{i=0}^k X_i^{p^i} + D_0.
\end{eqnarray*}
By construction $\phi[X,Y] = \phi(X) \phi(Y) - \phi(Y) \phi(X)$ for all $X, Y \in L_D$.  Furthermore,
every element of $L_{D_0} \subseteq U(L_{D_0})$ lies in the algebra generated by the image and so, by the
universal property of the enveloping algebra there is a surjective algebra homomorphism $\Phi : U(L_{D}) \twoheadrightarrow U(L_{D_0})$. 

To see that the map is injective we appeal to the graded algebra, as follows.
Suppose that $I = \Ker \Phi$ is a nonzero ideal of $U(L_D)$. Then $U(L_D) / I \cong U(L_{D_0})$ and, in particular,
their Gelfand--Kirillov dimensions coincide. By \cite[Prop. 8.1.15(iii)]{MR01} have $$\dim L = \GK U(L_{D_0}) = \GK (U(L_D)/I).$$
The PBW filtration on $U(L_D)$ induces a filtration on $U(L_D)/I$ and, according to Proposition 7.6.13 of {\it op. cit.}
we have $$\gr (U(L_D)/I) \cong \gr U(L_D) / \gr I \cong S(L_D) / \gr I.$$ Now Proposition~8.1.14 in \emph{op. cit.}
tells us that $$\GK (U(L_D)/I) = \GK (S(L_D ) / \gr I).$$ Since $S(L_D) / \gr I$ is a commutative affine algebra,
Theorem 8.2.14(i) in that same book tells us that $\GK (S(L_D) / \gr I)$ is equal to the Krull dimension of $S(L_D) / \gr I$,
which is necessarily less than $$\GK S(L_D) = \dim L_D = \dim L_{D_0},$$ since $S(L_D) / \gr I$ is a proper quotient. This contradiction
tells us that $I = 0$ as desired.
\end{proof}

\section{Calculating indexes}
We continue to let $\k$ have characteristic $p > 0$ and pick $k \geq 3$.
Let $\A = \langle x_1,..., x_k\rangle$ be an abelian Lie algebra and define a derivation of
$\A$ by $D(x_1) = x_i$ for all $i=1,...,{k-2}$, $D(x_{k-1}) = x_k$, $D(x_k) = 0$. We consider
the semidirect product $\A_D$. Now denote by $D_0$ the zero derivation of $\A_D$, and write
$D'$ for the derivation $D^p$ of $\A_D$. Define 
\begin{eqnarray*}
L:= (\A_D)_{D_0},\\
L' := (\A_D)_{D'}.
\end{eqnarray*}
According to Proposition~\ref{isoprop} we have $U(L) \cong U(L')$.
\begin{Lemma}\label{indexLemma}
We have
\begin{eqnarray*}
& & \ind L = k;\\
& & \ind L' \leq k-2.
\end{eqnarray*}
\end{Lemma}
\begin{proof}
Pick $\chi \in L^*$ and observe that $[L, L] = \langle x_1 x_2,...,x_{k-2}, x_k\rangle$, which implies that $L_\chi$
is completely determined by $(\chi(x_1), \chi(x_2) , ...,\chi(x_{k-2}), \chi(x_k)) \in \k^{k-1}$.
Choosing scalars $a_i, b_i, n_j, m_j\in \k$ with $i =1,...,k$ and $j = 1,2$ determines two elements of $L$:
\begin{eqnarray*}
X = \sum_{i} a_ix_i + n_1 D + n_2 D_0;\\
Y = \sum_i b_i x_i + m_1 D + m_2 D_0.
\end{eqnarray*}
Observe that 
\begin{eqnarray}\label{vaneq}
\chi[X, Y] = \sum_{i=1}^{k-2}(n_1b_i - m_1 a_i)\chi(x_i) + (n_1 b_{k-1} - m_1 a_{k-1})\chi(x_k).\medskip 
\end{eqnarray}
The assertion $X \in L_\chi$ is equivalent to saying that the right hand side of (\ref{vaneq})
vanishes for every choice of $b_i, m_j$ for $i=1,...,k$ and $j=1,2$. We shall use this observation
to show that $\dim L_\chi \geq k$ for all $\chi \in L^*$. If $\chi(x_i) \neq 0$ for some $i \in \{1,...,k-2\}$ 
then we may pick $Y$ by setting scalars $b_j = \delta_{i,j}$ and $m_1 = m_2 = 0$.
Now the vanishing of (\ref{vaneq}) ensures $n_1 = 0$. If $\chi(x_k) \neq 0$ then we may
pick $b_j = \delta_{j, k-1}$ and $m_1 = m_2 = 0$ to arrive at the conclusion $n_1 = 0$ similarly.

In either case, the assertion $X\in L_\chi$ is now equivalent to
$m_1 (\sum_{i=1}^{k-2} a_i \chi(x_i)) +  m_1 a_{k-1} \chi(x_k) = 0$. 
This final condition on $L_\chi$ is a single linear dependence
between the scalars $a_1,...,a_k$, and we conclude that $\dim L_\chi \geq k$
for all $\chi \in L^*$, and that this lower bound is attained whenever
$$(\chi(x_1), \chi(x_2) ..., \chi(x_{k-2}), \chi(x_k))\neq (0,0,...,0).$$
This proves $\ind L = k$.

Next observe that $[L', L'] = \langle x_1, x_2,...,x_{k-2}, x_k\rangle$, and so $\chi \in L'^*$ is
determined by  $(\chi(x_1), \chi(x_2) , ...,\chi(x_{k-2}), \chi(x_k)) \in \k^{k-1}$. We pick scalars as before so that
\begin{eqnarray*}
X = \sum_{i} a_i x_i + n_1 D + n_2 D';\\
Y = \sum_i b_i x_i + m_1 D + m_2 D'.
\end{eqnarray*}
are arbitrary elements of $L'$, and we have
$$\chi [X, Y] = \sum_{i=1}^{k-2}(n_1 b_i - m_1 a_i + n_2 b_i - m_2 a_i)\chi(x_i) + ( n_1 b_{k-1} - m_1 a_{k-1})\chi(x_k).$$
Now $X \in L'_\chi$ is equivalent to the vanishing of the right hand
side for every choice of $b_i, m_j$. It will suffice to exhibit $\chi$ such that $\dim L_\chi ' \leq k-2$.
To this end we take $\chi(x_1) = \chi(x_2) = \cdots = \chi(x_{k-2}) = \chi(x_k) = 1$.
Setting $b_j = \delta_{j, k-1}$ and $m_1 = m_2 = 0$ we obtain $n_1 = 0$,
whilst taking $b_j = \delta_{j,1}$ with $m_1 = m_2 =0$ subtends $n_2 = 0$.
Now take $b_i = 0$ for all $i$ and $m_1 = 0$, $m_2 = 1$ to get $\sum_{i=1}^{k-2} a_i = 0$.
Finally set $b_i = 0$ for all $i$ and $m_1 = 1$, $m_2 = 0$ to get $a_k = 0$. This shows that
$\dim L_\chi \leq k-2$, and we are done.
\end{proof}

\section{Proof of the theorem}

We let $L$ and $L'$ be the Lie algebras discussed in Lemma~\ref{indexLemma}. Our goal is to show that
(\ref{KWSt}) fails for $L$. Since the conjecture asserts that $M(L) = p^{\frac{1}{2}(\dim L - \ind L)} = p$ it
will suffice to show that $p^2 | M(L)$. By Proposition~\ref{isoprop} we know that $U(L) \cong U(L')$ which
implies that $M(L) = M(L')$. By Lemma~\ref{indexLemma} and \cite[\S 5.4, Remark 1]{PS99} there is a non-negative integer $s$ such that
$$p^{\frac{1}{2}(\dim L' - \ind L')} = p^{2 + s} | M(L').$$
$\hfill \qed$

\end{document}